\documentclass[11pt,reqno]{amsart}
\usepackage{amssymb,amsmath,amsthm,newlfont,enumerate}
\usepackage{a4wide}

\newcommand{\bbt}{\mathbb{T}}
\newcommand{\bbd}{\mathbb{D}}

\newcommand{\bbz}{\mathbb{Z}}

\newcommand{\calc}{\mathcal{C}}

\newcommand{\dsp}{\displaystyle}
\newcommand{\e}{\varepsilon}

\newtheorem{theo}{Theorem}[section]
\newtheorem{prop}[theo]{Proposition}

\newtheorem{lemme}[theo]{Lemma}

\begin{document}

\title[On powers of operators with spectrum in Cantor sets and spectral synthesis]{On powers of operators with spectrum in Cantor sets and spectral synthesis.}
\author[ M. Zarrabi]{Mohamed Zarrabi}
\address{Universit\'e de Bordeaux\\ Institut de Math\'ematiques de Bordeaux UMR 5251\\
351 cours de la Liberation\\F-33405 Talence Cedex\\France}
\email{Mohamed.Zarrabi@math.u-bordeaux.fr}


\date{}
\maketitle

\begin{abstract}
For $\xi \in \big( 0, \frac{1}{2} \big)$, let $E_{\xi}$ be the perfect symmetric set associated with $\xi$, that is
$$
E_{\xi} = \Big\{ \exp \Big( 2i \pi (1-\xi) \dsp \sum_{n = 1}^{+\infty} \epsilon_{n} \xi^{n-1} \Big) : \, \epsilon_{n} = 0 \textrm{ or } 1 \quad (n \geq 1) \Big\}
$$
and 
$$
\dsp b(\xi) = \frac{\log{\frac{1}{\xi}} - \log{2}}{2\log{\frac{1}{\xi}} - \log{2}}.
$$
Let $q\geq 3$ be an integer and $s$ be a nonnegative real number. We show that any invertible operator $T$ on a Banach space with spectrum contained in $E_{1/q}$ that satisfies 
\begin{eqnarray*} 
& & \big\| T^{n} \big\| = O \big( n^{s} \big), \,n \rightarrow +\infty \\
& \textrm{and} & \big\| T^{-n} \big\| = O \big( e^{n^{\beta}} \big), \, n \rightarrow +\infty \textrm{ for some } \beta < b(1/q),
\end{eqnarray*}
also satisfies the stronger property $\dsp \big\| T^{-n} \big\| = O \big( n^{s} \big), \, n \rightarrow +\infty.$   We also show that this result
is false for  $E_\xi$ when $1/\xi$ is not a Pisot number and that the constant $b(1/q)$ is sharp. As a consequence we prove that, if $\omega$ is a submulticative weight such that $\dsp \omega(n)=(1+n)^s, \, (n \geq 0)$ and $\dsp C^{-1} (1+|n|)^s \leq \omega(-n) \leq  C e^{n^{\beta}},\, (n\geq 0)$, for some constants $\dsp C>0$ and  $\dsp \beta < b( 1/q),$ then $E_{1/q}$ satisfies spectral synthesis in the Beurling algebra of all continuous functions $f$ on the unit circle $\bbt$ such that 
$\sum_{n = -\infty}^{+\infty} | \widehat{f}(n) | \omega (n) < +\infty$.

\end{abstract}

\section{Introduction.}

Let $T$ be an invertible operator defined on a Banach space with spectrum ($\mathrm{Sp}(T)$) contained in the unit circle $\bbt$. Several results show that  growth
conditions on $\dsp (\|T^n\|)_{n\geq 0}$ and  geometric or arithmetic conditions on $\dsp \mathrm{Sp}(T)$ induce  a certain behavior on the negative powers 
$\dsp (\|T^{-n}\|)_{n\geq 0}$. For example,  it is  shown by the author in \cite{Zar1},  that if $\dsp \sup_{n\geq 0} \|T^n\|<+\infty$ and if $\dsp \mathrm{Sp}(T)$ is countable then 
$\dsp \sup_{n\geq 0} \|T^{-n}\|<+\infty$ or $\dsp \limsup_{n\to +\infty}\frac{\log  \|T^{-n}\|}{\sqrt{n}}>0$. See also \cite[Theorem 5.6]{BY1}  
for a similar result on representations of groups by operators.

For $\xi \in \big( 0, \frac{1}{2} \big)$, Let $E_{\xi}$ be the perfect symmetric set associated with $\xi$, that is
$$
E_{\xi} = \Big\{ \exp \Big( 2i \pi (1-\xi) \dsp \sum_{n = 1}^{+\infty} \epsilon_{n} \xi^{n-1} \Big) : \, \epsilon_{n} = 0 \textrm{ or } 1 \quad (n \geq 1) \Big\},
$$
and let
$$
\dsp b(\xi) = \frac{\log{\frac{1}{\xi}} - \log{2}}{2\log{\frac{1}{\xi}} - \log{2}}.
$$
Notice that $E_{1/3}$ is the classical triadic Cantor set. Let $q\geq 3$ be an integer. J. Esterle showed  in \cite[p. 79]{Est2} that if 
$\dsp \sup_{n\geq 0} \|T^n\|<+\infty$ and $\dsp \mathrm{Sp}(T)\subset E_{1/q}$ then 
$\dsp \sup_{n\geq 0} \|T^{-n}\|<+\infty$ or $\dsp\limsup_{n\to +\infty}\frac{\log  \|T^{-n}\|}{{n}^\beta}=+\infty$, for every $\beta< b(1/q)$. On the other hand it is proved in \cite[Theorem 2.2 ]{ERZ} that the constant $b(1/q)$ is sharp. For  further results of this type see  \cite{Agr1}, \cite{Agr2}, \cite{AK}, \cite{Est1}, \cite{Est2}, \cite{FK}, \cite{Zar1}. In this paper we are interested by operators with spectrum contained in $E_\xi$ and such that $\dsp (\|T^n\|)_{n\geq 0}$ has a polynomial growth. Let $s$ be a nonnegative real number. We show (Theorem \ref{operateur})    that if 
$$
\big\| T^{n} \big\| = O \big( n^{s} \big), \,n \rightarrow +\infty,
$$
$$
\big\| T^{-n} \big\| = O \big( e^{n^{\beta}} \big), \, n \rightarrow +\infty \textrm{ for some } \beta < b(1/q),
$$
and if $\dsp \mathrm{Sp}(T)\subset E_{1/q}$,  then $T$ satifies the stronger condition
$$
\big\| T^{-n} \big\| = O \big( n^{s} \big), \,n \rightarrow +\infty.
$$
This extends the result in    \cite[p. 79]{Est2} for  the case $s=0$  and improves significantly  \cite[Corollary 2.5]{Agr2}, where it is proved that, under the same assumptions, 
$$
\big\| T^{-n} \big\| = O \big( n^{r} \big), \,n \rightarrow +\infty \ {\text{ for every }}\ r>s+1/2.
$$
On the other hand we show (Propositions \ref{sharp 1} and \ref{sharp 2}) that the above result  is false for $E_\xi$ when $1/\xi$ is not a Pisot number and that the constant $b(1/q)$ is sharp.

In section 3, we apply   Theorem \ref{operateur}   to show  that the sets $E_{1/q}$ satisfy  spectral synthesis  in  certain  Beurling weighted algebras. First, we recall  some notations and definitions.  Let $\omega = \big( \omega(n) \big)_{n \in \bbz}$ be a weight, that is a real sequence such that $\omega(n) \geq 1$ and $\omega(n+m) \leq \omega(n) \omega(m)$ for $m,n\in\bbz$. We define the Banach algebra
$$
A_{\omega}(\bbt) = \Big\{ f \in \calc(\bbt) : \, \big\| f \big\|_{\omega} = \sum_{n = -\infty}^{+\infty} | \widehat{f}(n) | \omega(n) < +\infty \Big\},
$$
where $\widehat{f}(n)$ is the $\textrm{n}^{\textrm{th}}$ Fourier coefficient of $f$, and $\calc(\bbt)$ the space of all continuous functions on $\bbt$. In this paper, we will consider only regular algebras, that is algebras $A_{\omega}(\bbt)$ for which the weight satisfies $\dsp \sum_{n \in \bbz} \frac{\log{\omega(n)}}{1+n^{2}}<+\infty$ (see \cite[ex. 7, p. 118]{Katz}). 

Let $E$ be a closed subset of $\bbt$. We define the closed ideal
$$ 
I_{\omega}(E) = \Big\{ f \in A_{s}(\bbt) : \, f = 0 \ {\text{ on }} \ E \Big\},
$$
and the ideal $J_{\omega}(E)$ to be the closure in $A_{\omega}(\bbt)$ for the norm $\| \, \|_{\omega}$ of the set
$$
\Big\{ f \in A_{\omega}(\bbt) : \,  f = 0 \ {\text{on a neighborhood of}}\ E \Big\}.
$$
We say that a function $f$ satisfies spectral synthesis for $E$ in $A_{\omega}(\bbt)$ if $f \in J_{\omega}(E)$, and that $E$ satisfies spectral synthesis in $A_{\omega}(\bbt)$ if 
$J_{\omega}(E) = I_{\omega}(E).$

Let $I$ be a closed ideal of $A_{\omega}(\bbt)$ such that $h(I)=E$, where 
$$
h(I)=\{z\in\bbt :\ f(z)=0,\ \forall \ f\in I\}.
$$
Then $J_{\omega}(E) \subset I\subset I_{\omega}(E)$ (see \cite[p. 224]{Katz}). Thus the set $E$ satisfies spectral synthesis in $A_{\omega}(\bbt)$ if and only  if every closed ideal $I$ of $A_{\omega}(\bbt)$ such that
 $h(I)=E$ is equal to $I_{\omega}(E)$.
 
In the  case where the weight $\omega$ is defined by $\omega(n) = (1+|n|)^{s}$ for all $n \in \bbz$, with $s$ a nonnegative real number, we denote by $\big( A_{s}(\bbt), \| \, \|_{s} \big)$ the algebra $\big( A_{\omega}(\bbt), \| \, \|_{\omega} \big)$, and by $I_{s}(E)$ and $J_{s}(E)$ the closed ideals $I_{\omega}(E)$ and $J_{\omega}(E)$. We remark that $A_{0}(\bbt) = A(\bbt)$ is in fact the classical Wiener algebra. 

In \cite{Her}, C. S. Herz proved that  $E_{1/q}$ satisfies spectral synthesis in the Wiener algebra $A(\bbt)$ (see also \cite[p. 58--59]{Kaha1}). In \cite{Est2}, J. Esterle showed  that $E_{1/q}$ satisfies spectral synthesis in $A_{\omega}(\bbt)$ for all weights $\omega$ such that $\omega(n)=1 \, (n \geq 0)$ and $ \omega(-n) = O \big( e^{n^{\beta}} \big)$, $n\to+\infty$, for some $\dsp \beta < b( 1/q).$ Here we extend this result (Theorem \ref{synthesecantor})   to a  larger class  of weights, that is,  weights satisfying 
$$
\omega(n) = (1+n)^{s} \quad (n \geq 0),
$$
and
 $$\dsp C^{-1} (1+|n|)^s \leq \omega(-n) \leq  C e^{n^{\beta}},\quad (n\geq 0),$$
for some constants $s\geq 0$, $\dsp C>0$ and  $\dsp \beta < b( 1/q).$

\section{Behavior of powers of operators with spectrum in Cantor sets.}
Let $\xi \in \big( 0, \frac{1}{2} \big)$. We write  $\bbt\setminus E_\xi= \bigcup_{n\geq 0}L_n$, where the $L_n$ are the contiguous arcs to  $E$. We know from the construction of  $E_\xi$ (see \cite[Chapitre I]{KaSa}) that for every $k\geq 1$, there are exactly  $2^{k-1}$ arcs among the  $L_n$, which are of length  $2\pi\xi^{k-1}(1-2\xi)$. Hence, for a real number $\gamma$,
$$
\sum_{n\geq 0}|L_n|^\gamma=\sum_{k\geq 1}2^{k-1}\left(2\pi\xi^{k-1}(1-2\xi)\right)^\gamma={{(2\pi(1-2\xi))^\gamma}\over{2\xi^\gamma}}\sum_{k\geq 1}(2\xi^\gamma)^k.
$$
\noindent So, the sum  $\sum_{n\geq 0}|L_n|^\gamma$  is finite if and only if  $\gamma>{{\log 2}\over{\log {1/\xi}}}$. On the other hand, by a simple computation, we can see  that for a real number $\delta$, the integral  
$\dsp \int_0^{2\pi}{{1}\over{d(e^{it},E_\xi)^\delta}}dt$ converges if and only if  the series  $\sum_{n\geq 0}|L_n|^{1-\delta}$ converges. Therefore
\begin{eqnarray} \label{integral}
\int_0^{2\pi}{{1}\over{d(e^{it},E_\xi)^\delta}}dt<+\infty \  {\textrm{ if and only if }} \ \delta < 1-{{\log 2}\over{\log {1/\xi}}}.
\end{eqnarray}
We denote by  $A^\infty (\bbd)$ the set of all functions holomorphic in the unit disk $\bbd$ and infinitely differentiable in the closed disk $\overline{\bbd}$.

\begin{lemme} \label{lem op}
Let $s$ be a nonnegative real, $q \geq 3$ be an integer and  $\omega$ be a weight such that 
$$\dsp \omega (n)=(1+n)^s,\ n\geq 0,$$
and
$$
\dsp {\omega(-n)}=O(e^{n^\beta}),\ n\to +\infty,  \ {\text{for some}} \ \beta < b(1/q).
$$ 
 
 Then there exists an outer fonction $f\in A^\infty (\bbd)$ such that $f(0)=1$ and for every nonnegative integer $m$, $f(z^{q^m})\in J_\omega (E_{1/q})$.
\end{lemme}

\begin{proof}
Let $\gamma$ be a real number such that  $\beta<\gamma <b(1/q)$. There exists a real number $\delta$ such that  $\frac{\gamma}{1-\gamma}<\delta <1-{{\log 2}\over{\log {q}}}$. We set 
$$\dsp \omega_\gamma (n)=(1+n)^s,\ n\geq 0 \ \ \mathrm{ and }\ \ \omega_\gamma (-n)=e^{n^\gamma},\ n>0.$$
 Since  $\delta <1-{{\log 2}\over{\log {q}}}$, it follows from  (\ref{integral})  and  \cite[Lemma 7.2]{Est2}  that there exists a nonzero outer function $f\in A^\infty (\bbd)$ which vanishes exactly on $E_{1/q}$ and which satisfies 
$$
|f(z) |\leq e^{- d(z,E_{1/q})^{-\delta}}, \ z\in \bbt\setminus E_{1/q}.
$$
We may assume that $f(0)=1$. Now since  $\frac{\gamma}{1-\gamma}<\delta$, by \cite[Lemma 7.1]{Est2}, we have $f\in J_{\omega_\gamma} (E_{1/q})$. 

Let $m$ be a nonnegative integer. Notice that for every $z\in E_{1/q}$, $z^{q^m}\in E_{1/q}$. Take a sequence  $(f_n)_n\subset A_{\omega_\gamma}(\bbt)$, of functions vanishing on a neighborhood of $E_{1/q}$ and such that $\Vert f_n-f\Vert_{\omega_\gamma}\to 0$, $n\to \infty$. For every $n$, $ f_n(z^{q^m})$ vanishes on a neighborhood of   $E_{1/q}$ and   we have
$$
\begin{aligned}
\Vert f_n(z^{q^m})-f(z^{q^m})\Vert_{\omega} &= \sum_{k\in\bbz} |\hat{f_n}(k)-\hat{f}(k)|\omega (q^mk) \\
&\leq \sup_{k\in\bbz}\frac{\omega (q^mk) }{\omega_\gamma (k)} \Vert f_n-f\Vert_{\omega_\gamma} \\
&\leq C\Vert f_n-f\Vert_{\omega_\gamma},
\end{aligned}
$$
where $C=\sup_{k\in\bbz}\big( \omega(k)/\omega_\beta (k)\big) \max\left(q^{ m s}, \sup_{k\geq 0}\big(e^{q^{\beta m}k^\beta-k^\gamma}\big)\right)<+\infty$, where $\omega_\beta$ is defined similarly to $\omega_\gamma$. Thus $\Vert f_n(z^{q^m})-f(z^{q^m})\Vert_{\omega} \to 0$, as $n\to \infty$.
\end{proof}

The following theorem is the main result of this section.

\begin{theo} \label{operateur}
Let $s$ be a nonnegative real number, $q \geq 3$ an integer  and $T$ an  invertible operator  on a Banach space with spectrum contained in $E_{1/q}$ . If 
$$
\big\| T^{n} \big\| = O \big( n^{s} \big), \,n \rightarrow +\infty,
$$
and
$$
 \big\| T^{-n} \big\| = O \big( e^{n^{\beta}} \big), \, n \rightarrow +\infty \textrm{ for some } \beta < b(1/q),
$$
Then $T$  satisfies the stronger property 
$$\big\| T^{-n} \big\| = O \big( n^{s} \big), \, n \rightarrow +\infty.$$
\end{theo}

\begin{proof} We take the weight $\omega (n)=(1+n)^s$ if $n\geq 0$ and $\omega (n)= e^{|n|^\beta}$ if $n<0$.  For $f \in A_{\omega}(\bbt) $, we set
$$
f(T)= \sum_{n=-\infty}^{+\infty} \widehat{f}(n) T^{n}.
$$
Let $I$ be the set of functions $f \in A_{\omega}(\bbt) $ such that $f(T)=0$. Then $I$ is a closed ideal of $A_{\omega}(\bbt)$ and $h(I)=\mathrm{Sp}(T)$ 
(see \cite[Th\'eor\`eme 2.4]{Zar1}). Therefore $J_{\omega} (E_{1/q})\subset I$. Now, let $f$ the outer function given by Lemma \ref{lem op}. Then, for every nonnegative integer $m$, $f(T^{{q^m}})=0$.

 Let $n$ be an integer greater or equal to 1 and let  $m$ be the unique integer such that  $q^m\leq n<q^{m+1}$. It follows from the last equalities that 
$$
T^{-n}=-\sum_{k=1}^{+\infty}\hat f(k)T^{kq^{m+1}-n},
$$
and 
$$
\begin{aligned}
\|T^{-n}\| &\leq\sum_{k=1}^{+\infty}|\hat f(k)|\|T^{kq^{m+1}-n}\| \\
&\leq \sum_{k=1}^{+\infty}|\hat f(k)|(1+kq^{m+1}-n)^s \\
&\leq q^sn^s\sum_{k=1}^{+\infty}|\hat f(k)|(1+k)^s.
\end{aligned}
$$

\end{proof}

The remainder of this section is devoted to discussing the sharpness of the hypothesis in  
Theorem \ref{operateur}. In the following proposition we show that the constant $b(\xi)$ is the best possible.

\begin{prop}\label{sharp 1} For every $\xi \in (0,1/2)$ there exists a contraction $T$ on a Hilbert space with spectrum contained in $E_\xi$, sucth that 
$$
\big\| T^{-n} \big\| = O \big( n^{b(\xi)} \big), \,n \rightarrow +\infty,
$$
and
$$\limsup_{n\to+\infty}\| T^{-n} \|/e^{n^\beta} = +\infty, \  \textrm{ for every } \ \beta < b(\xi).$$

\end{prop}

\begin{proof}
Let $L$ be the classical  Lebesgue function associated to $E_\xi$ and let $dL$ be 
 the unique positive measure  on $[0, 2\pi]$ such that $\int_a^bdL=L(b)-L(a)$ for $0\leq a\leq b\leq 2\pi$ (see \cite[Chapter 1]{KaSa} and \cite{ERZ}). 
 We set
$$
V(z)=\exp \biggl({1\over 2\pi}\int_0^{2\pi}\frac{z+e^{it}}{z-e^{it}}dL(t)\biggr) \ \ (|z|<1),
$$
the inner function associated to the measure $dL$.  Denote  by $H^2$ and $H^\infty$ the usual Hardy spaces on the unit disk $\bbd$. 
Set ${\mathcal {H}}=H^2\circleddash VH^2$. Let $P_{{\mathcal {H}}}$ be the orthogonal projection from $H^2$ onto ${\mathcal {H}}$ and $T :{\mathcal {H}}\to {\mathcal {H}}$, $g\to P_{{\mathcal {H}}}(\alpha g)$, where $\alpha$ is the identity function $z\to z$; T is a contraction and its spectrum is contained in $E_\xi$. It is shown in \cite[Theorem 2.2]{ERZ} that $\big\| T^{-n} \big\| = O \big( n^{b(\xi)} \big), \,n \rightarrow +\infty$ and $\limsup_{n\to\infty}\|T^{-n}\|=\infty$. 

Assume  that for somme $\beta < b(\xi)$, $\sup_{n\geq 0}\| T^{-n} \|/e^{n^\beta}<+\infty$. We set 
$$\dsp \omega_\beta (n)=1,\ n\geq 0 \ \ \mathrm{ and }\ \ \omega_\beta (-n)=e^{n^\beta},\ n>0.$$
As in the proof of Lemma \ref{lim0},   Using (\ref{integral}),  \cite[Lemma 7.2 and Lemma 7.1]{Est2} we show that there exists a nonzero outer fonction $f\in A^\infty (D)$ such that  $f\in J_{\omega_\beta} (E_{1/q})$. By the same argument as in the proof of Theorem \ref{operateur}, we get that $f(T)=0$. On the other hand  the operator $f(T)$ is defined by $g\to  P_{{\mathcal {H}}}(f g)$. Finaly, for $g=\frac{V-V(0)}{z}\in\mathcal{H}$, the equality $f(T)g=0$ implies that $f\in VH^2$.


\end{proof}

Now we will  prove that   Theorem \ref{operateur} cannot be extended to all the sets $E_\xi$. More precisely we show that it is false for all $E_\xi$ such that $1/\xi$ is not a Pisot number. For this we introduce the definition below.
We assume that $s$ is a nonnegative integer. We set  $A_s^+(\bbt)=\{f\in A_s(\bbt) :  \ \hat f (n)=0, \ n<0\}$.  A closed subset $E$ of $\bbt$ is said to be an interpolating set of order  $s$ for $A_s(\bbt)$, if for every function $f\in A_s(\bbt)$, there exists  $g\in A_s^+(\bbt)$ such that 
$$
f^{(k)}(z) = g^{(k)}(z), \ \ z\in E, \ \ 0\leq k\leq s. 
$$

 It is shown in \cite{Zar2}, that $E_\xi$ is an interpolating set of order  $s$ for $A_s
 (\bbt)$ if and only if ${{1}/{\xi}}$ is a  Pisot number.
 
\begin{prop}\label{sharp 2} Assume that $s$ is a nonnegative integer and that $1/\xi$ is not a Pisot number. Then there exists a bounded operator $T$ on a Banach space such that 
$$
\big\| T^{n} \big\| = O \big( n^{s} \big), \,n \rightarrow +\infty,
$$
$$
\forall \ r>s+{\frac {3}{2}}{\frac {\log 2}{\log {1/\xi}}},
\  \big\| T^{-n} \big\| = O \big( n^{r} \big), \, n \rightarrow +\infty,
$$
and 
$$\limsup_{n\to+\infty}\| T^{-n} \|/n^s = +\infty.$$

\end{prop}

\begin{proof}
 We denote by  $\pi$ the canonical  surjection from $A_p^+(\bbt)$ onto $A_p^+(\bbt)/I_s^+(E_\xi)$, where 
 $I_s^+(E_\xi)=\{f\in A_s(\bbt),\ f=0\ \mathrm{on}\ E_\xi\}$. Note that since $E_\xi$ is a perfect set $I_s^+(E_\xi)$ is also the set of all functions $f$ in $A_p^+(\bbt)$ such that for all $0\leq k\leq s$, $f^{(k)}=0$ on $E_\xi$.  Let $\alpha$ be the identity function $z\to z$. It is easily seen that $E_\xi$ is an interpolating set of order  $s$ for  $A_s^+(\bbt)$ if and only if we have  
$$
\|\pi(\alpha)^{-n}\|_p=O(n^s),\ \ \ n\to +\infty.
$$
Let $T$ be the operator defined on $A_p^+(\bbt)/I_s^+(E_\xi)$ by $\pi (f)\to \pi(\alpha f)$. We have $\mathrm{Sp} (T)=E_\xi$ and $\big\| T^{n} \big\| =\big\| \pi(\alpha)^{n} \big\|= O \big( n^{s} \big), \, n \rightarrow +\infty.$

It is shown in \cite[Theorem 6]{Zar2} that $E_\xi$ is not  an interpolating set of order  $s$ for  $A_s^+(\bbt)$. Thus 
$$
\limsup_{n\to+\infty}\| T^{-n} \|/n^s = \limsup_{n\to+\infty}\| \pi(\alpha)^{-n} \|/n^s = +\infty.
$$
Since $E_\xi$ satisfies condition (\ref{integral}) it follows from the proof of \cite[Theorem 8]{Zar2} that for every $\delta<1- \frac{\log 2}{\log {1/\xi}}$ and every $r>s+\frac{3}{2}(1-\delta)$,
$$
\big\| \pi(\alpha)^{-n} \big\|= O \big( n^{r} \big), \ n \rightarrow +\infty.
$$
Thus for every $r>s+{\frac {3}{2}}{\frac {\log 2}{\log {1/\xi}}}$,
 $\big\| T^{-n} \big\| = O \big( n^{r} \big), \, n \rightarrow +\infty.$
 
\end{proof}

\section{Spectral synthesis for Cantor sets in weighted Beurling algebras.} 

Let $s$ be a nonnegative real number. We denote by $[s]$ the integral part of $s$, that is  the nonnegative integer such that $[s] \leq s < [s]+1$. For $f$ in $A_{s}(\bbt)$, we denote by   $f^{(j)}$, $0\leq j\leq [s]$,  the $j^{th}$--derivative of $f$  with respect to the variable $t$, that is $f^{(j)} \big( e^{it} \big)=\frac{\partial^j}{\partial t^j}\big(f\big( e^{it} \big)\big).$

To prove the main theorem of this section (Theorem \ref{TheorSynthese}), we show first that 
  $E_{{1}/{q}}$ satisfies spectral synthesis in $A_{s}(\bbt)$. To do this,  we approximate any function $f$ in $A_{s}(\bbt)$ such that $f^{(j)}(1)=0$ $(0\leq j \leq [s])$, by a sequence of functions $(f_{N,[s]})_{N \geq 1}$ in $A_{s}(\bbt)$ which are piecewise polynomial and such that each function $f_{N,[s]}^{([s])}$ interpolates $f^{([s])}$ at the points $\dsp e^{\frac{2ik \pi}{N}}$, for all $k \in \{0, \ldots , N-1\}$ (for $s=0$, see the Herz criterion in \cite{Her} and \cite[p. 58--59]{Kaha1}). The functions $f_{N,[s]}$ are  defined by
\begin{eqnarray}
f_{N,[s]}(e^{it}) =  \sum_{k=0}^{N-1} \Delta_{[s],\frac{2 \pi}{N}} \big( t - \frac{2k \pi}{N} \big) f^{([s])} \big( e^{\frac{2ik \pi}{N}} \big), \label{deffn}
\end{eqnarray}
where $\Delta_{[s], \e}$ (for  $0 < \e \leq \pi$) is the  $2 \pi$-periodical function given on $[- \pi , \pi]$ by the formula
\begin{eqnarray*}
\Delta_{[s], \e}(t) = \left\{ \begin{array}{ll}
\dsp (-1)^{[s]}\frac{(\e - t)^{[s]+1}}{([s]+1)! \, \e}    & \textrm{ for } t \in [0,\e] \\
0 & \textrm{ for } t \in [\e, \pi] 
\end{array} \right. ,
\end{eqnarray*}
and $$\Delta_{[s], \e}(t) = (-1)^{[s]} \Delta_{[s], \e}(-t) \ \textrm{ for }\  t \in [-\pi,0].$$

\begin{lemme} \label{controle} Assume that $0\leq s <1$. 
There exists a positive constant $K$ such that for every function $f\in A_{s}(\bbt)$      and every integer $N \geq 1$,   
$$
\| f_{N,0} \|_{s} \leq K \| f \|_{s}, \quad \big( f \in A_{s}(\bbt) \big).
$$
\end{lemme}

\begin{proof}First, we will show this inequality for $f(e^{it}) = e^{int} \, (n \in \bbz)$. We have $[s]=0$ and 
\begin{eqnarray*}
\Delta_{0, \e}(t) = \left\{ \begin{array}{ll}
\dsp \frac{(\e - |t|)}{ \e}    & \textrm{ for } t \in [-\e,\e] \\
0 & \textrm{ for } t \in [\e, \pi] 
\end{array} \right. 
\end{eqnarray*}
The Fourier coefficients of $\Delta_{0, \e}$ are : 
$$ 
\widehat{\Delta_{0,\e}}(m)  = \frac{\e}{2 \pi} \frac{\sin^2 (m\e/2)}{(m\e/2)^2} \, \textrm{ if } \, m\ne 0,
 \, \textrm{ and } \, \widehat{\Delta_{0,\e}}(0)=\frac{\e}{2 \pi}.
$$
Let $f(e^{it})= e^{int}$. For $N \geq 1$,  we set 
$$D_{N} = \dsp \sum_{k=0}^{N-1} f \big( e^{\frac{2ik \pi}{N}} \big) \delta_{\frac{2k \pi}{N}}=\dsp \sum_{k=0}^{N-1}  e^{\frac{2ikn \pi}{N}}  \delta_{\frac{2k \pi}{N}},
$$
where $\delta_{a}$ denotes the Dirac measure at $a$. We have
\begin{eqnarray*}
\widehat{D_{N}}(m) = \left\{ \begin{array}{ll}
\dsp \frac{N}{2 \pi}  & \textrm{ if } m = n + k N \quad (k \in \bbz) \\
0 & \textrm{ otherwise }
\end{array} \right. 
\end{eqnarray*}
Therefore
\begin{eqnarray}\label{fourier coeff}
\widehat{f_{N,0}}(m)  =  2 \pi  \widehat{D_{N}}(m) \widehat{\Delta_{0,\frac{2 \pi}{N}}}(m) =
 \left\{ \begin{array}{lll}
 \dsp \frac{\sin^2 (m\pi/N)}{(m\pi/N)^2} & \textrm{ if } m = n + k N \ne 0 \ \textrm{ for\ some }k \in \bbz \\
 1 & \textrm{ if } m = n + k N = 0 \ \textrm{ for\ some }k \in \bbz \\
 0 & \textrm{ otherwise }
\end{array} \right.
\end{eqnarray}

We will consider two cases :

\hspace*{0.5cm} Case 1: $N \geq 2|n|$. \\

Let $m = n + k N$ for some $k \in \bbz$. Note that since $N\geq 2|n|$, $m\ne 0$. Thus we have 
\begin{eqnarray}\label{major-1}
|\widehat{f_{N,0}}(m)|(1+|m|)^{s} & = & \frac{\sin^2 \big((n/N+k)\pi)\big)}{(n/N+k)^2\pi^2}(1+|n + k N|)^{s} \nonumber\\
 & = & \frac{\sin^2 \big((n/N)\pi)\big)}{(n/N+k)^2\pi^2}(1+|n + k N|)^{s} \nonumber\\
& \leq & \frac{(n/N)^{2-s} }{(|k|-1/2)^2}(|n|/N+n^2/N + |kn|)^{s} \nonumber\\
& \leq & \frac{1}{2^{2-s}} \frac{1}{(|k|-1/2)^2}(1/2+|n|/2 +  |kn|)^{s} \nonumber\\
& \leq & \frac{1}{2^{2-s}}\frac{(|k|+1/2)^s}{(|k|-1/2)^2} (1+|n|)^s.
\end{eqnarray}
Therefore 
\begin{eqnarray} \label{major0}
\dsp\big\| f_{N,0} \big\|_{s} & = & \sum_{k = -\infty}^{+\infty} \big| \widehat{f_{N,0}}(n + k N) \big| (1+|n + k N|)^{s} \nonumber\\
& \leq & \frac{1}{2^{2-s}}\sum_{k = -\infty}^{+\infty} \frac{(|k|+1/2)^s}{(|k|-1/2)^2}  (1+|n|)^{s} \, = \, \frac{1}{2^{2-s}}\sum_{k = -\infty}^{+\infty} \frac{(|k|+1/2)^s}{(|k|-1/2)^2}  \big\| f \big\|_{s}.
\end{eqnarray}

\hspace*{0.5cm} Case 2: $N \leq 2 |n|$. \\

For $|m|<N$,
$$
\big|\widehat{f_{N,0}}(m) \big| (1+|m|)^{s} \leq (1+N)^{s}\leq 2^s (1+|n|)^{s}.
$$
We remark that there are at most two integers $m$ of the form $m = n + k N$ with $|m| < N$. Thus 
\begin{eqnarray} \label{major1}
\sum_{|m|<N}\big| \widehat{f_{N,0}}(m) \big| (1+|m|)^{s} & \leq 2^{1+s} (1+|n|)^{s}.
\end{eqnarray}

Assume now that $|m|\geq N$ with  $m = n + k N$ and  $k \in \bbz$. We note that 
 in this case $\dsp \big| \frac{n}{N} + k \big| = \frac{|m|}{N} \geq 1$. We have 
 \begin{eqnarray*}
 \widehat{f_{N,0}}(m) \big| (1+|m|)^{s} & = & \frac{\sin^2 (n\pi/N)}{(n/N+k)^2\pi^2}(1+|n + k N|)^{s}\\
 & \leq &  \frac{1}{(n/N+k)^2\pi^2}N^s (1/N+ |n/N+k|)^s \\
 & \leq &  \frac{2^{2s}}{\pi^2|n/N+k|^{2-s}} |n|^s.
 \end{eqnarray*}
 Then we get 
 \begin{eqnarray} \label{major2}
 \sum_{|m|\geq N}\big| \widehat{f_{N,0}}(m) \big| (1+|m|)^{s} & \leq &  \sum_{k\in\bbz : \, |n + k N| \geq N} \widehat{f_{N,0}}(n + k N) \big| (1+|n + k N|)^{s} \nonumber\\
& \leq &  \sum_{k\in\bbz : \, |n/N + k | \geq 1} \frac{2^{2s}}{\pi^2|n/N+k|^{2-s}} |n|^s \nonumber\\
& \leq &  \frac{2^{2s+1}}{\pi^2}  \sum_{k\geq 1} \frac{1}{k^{2-s}} |n|^s. 
\end{eqnarray}
 Combining (\ref{major1}) and (\ref{major2}) we obtain 
 
 \begin{eqnarray} \label{major3}
\|f_{N,0}\|_s\leq \max\left(2^{1+s},\frac{2^{2s+1}}{\pi^2}  \sum_{k\geq 1}  \frac{1}{k^{2-s}} \right)(1+|n|)^{s}= \max\left(2^{1+s},\frac{2^{2s+1}}{\pi^2}  \sum_{k\geq 1}  \frac{1}{k^{2-s}}\right)\|f\|_s
 \end{eqnarray}
Finally, we deduce from (\ref{major0}) and (\ref{major3}) that there exists a positive constant $K$  such that for all integers $N \geq 1$ and $n\in\bbz$,
$$\|f_{N,0}\|_s\leq K \|f\|_s,$$
where $f(e^{it}) = e^{int}$. We deduce easily that this inequality is still true for all trigonometric polynomials, then, by density, for all functions $f$ in $A_{s}(\bbt)$. 
 
\end{proof} 
 
 \begin{lemme} \label{limitpoly}  Assume that $0\leq s <1$. Then
for every function $f\in A_{s}(\bbt)$, we have 
\begin{eqnarray*}
\lim_{N \rightarrow +\infty} \big\| f_{N,0} - f \big\|_{s} = 0.
\end{eqnarray*}
\end{lemme}

\begin{proof}
 
 First, we prove the proposition for $f(e^{it}) = e^{int}$ ($n \in \bbz$). If $n=0$, we deduce from (\ref{fourier coeff})  that for all $N \geq 1$, $\widehat{f_{N,0}} (0) = 1$ and $\widehat{f_{N,0}} (m) = 1$ for $m\ne 0$, so that $f_{N,0} = f$.
 
  Now, we suppose that $n \neq 0$. Let $N \geq 2|n|$ be fixed and $m = n + k N$ ($k \in \bbz$). If $k \neq 0$, then $\dsp |m| \geq \big( |k| - \frac{1}{2} \big) N$. We have
\begin{eqnarray}\label{lim0}
|\widehat{f_{N,0}}(m)|(1+|m|)^{s} & = & \frac{\sin^2 \big((m/N)\pi)\big)}{(m/N)^2\pi^2}(1+|m|)^{s} \nonumber\\
 & = & \frac{\sin^2 \big((n/N)\pi)\big)}{(m/N)^2\pi^2}(1+|m|)^{s} \nonumber\\
& \leq & \frac{2^s n^2 }{m^{2-s}} \nonumber\\
& \leq & \frac{2^s n^2 }{(|k|-1/2)^{2-s}N^{2-s}}.
\end{eqnarray} 

On the othar hand, if $k =0$, we have $m=n$ and 
\begin{eqnarray*}
 |\widehat{f_{N,0}}(n)-1| & = & \left|\frac{\sin^2 \big((n/N)\pi)\big)}{(n/N)^2\pi^2}-1\right| \\
  & \leq & \frac{\pi^2}{3} (n/N)^2. 
\end{eqnarray*}  
Thus 
\begin{eqnarray}\label{lim1}
 |\widehat{f_{N,0}}(n)-1| (1+|n|)^s \leq \frac{\pi^2}{3} (n/N)^2(1+|n|)^s \leq \frac{\pi^2}{3}\frac{n^2}{N^{2-s}}.
\end{eqnarray}
It follows from (\ref{lim0}) and (\ref{lim1}) that 
\begin{eqnarray*}
\big\| f_{N,0} - f \big\|_{s}  \leq \left( \frac{\pi^2 n^2}{3} + \sum_{k\in{\bbz\setminus \{0\}}}\frac{2^s n^2 }{(|k|-1/2)^{2-s}}\right) \frac{1}{N^{2-s}},
\end{eqnarray*}
which implies
\begin{eqnarray*}
\lim_{N \rightarrow +\infty} \big\| f_{N,0} - f \big\|_{s} = 0. 
\end{eqnarray*}
Now we consider the general case: let $f$ be in $A_{s}(\bbt)$ and $\e > 0$. There exists $M_{0} > 0$ such that $\dsp \big\| f - g \big\|_{s} \leq \e$, where $\dsp g = \sum_{m=-M_{0}}^{M_{0}} \widehat{f}(m) e^{imt}$. Using the above, we can find $N_{0} > 0$ such that for all $N \geq N_{0}$, $\dsp \big\| g - g_{N,0} \big\|_{s} \leq \e$. For $N \geq N_{0}$, we deduce from lemma~\ref{controle} that
\begin{eqnarray*}
\big\| f - f_{N,0} \big\|_{s} & \leq & \big\| f - g \big\|_{s} + \big\| g - g_{N,0} \big\|_{s} + \big\| g_{N,0} - f_{N,0} \big\|_{s} \\
& \leq & (1+K) \big\| f - g \big\|_{s} + \big\| g - g_{N,0} \big\|_{s} \\
& \leq & (2 + K) \e,
\end{eqnarray*}
which proves that 
\begin{eqnarray*}
\lim_{N \rightarrow +\infty} \big\| f - f_{N,0} \big\|_{s} = 0.  
\end{eqnarray*}
\end{proof} 
 
 \begin{lemme} \label{norme}  Let $s$ be a nonnegative real number and $p=[s]$. If $p\geq 1$, then 
for every function $f\in A_{s}(\bbt)$, we have 
\begin{eqnarray*}
\|f^{(p)}\|_{s-p}\leq \| f \|_{s} \leq \left(2^p+\frac{{(2\pi)^p}}{ (p+1)!}\right) \|f^{(p)}\|_{s-p} + \sum_{j=0}^{p-1}\frac{(2\pi)^j}{(j+1)!} |f^{(j)}(1)|.
\end{eqnarray*}
\end{lemme}

\begin{proof} For $f\in A_{s}(\bbt)$, we have 
 
\begin{eqnarray*} 
 f^{(p)}(e^{it})=\sum_{k\in\bbz} (ik)^p \widehat{f}(k) e^{ikt},
 \end{eqnarray*}
 and 
\begin{eqnarray*} 
\|f^{(p)}\|_{s-p}=\sum_{k\in\bbz}|k|^p (1+|k|)^{s-p} |\widehat{f}(k)|.
\end{eqnarray*} 
Thus 
\begin{eqnarray} \label{norme1}
\|f^{(p)}\|_{s-p}\leq \| f \|_{s} \leq |\widehat{f}(0)|+ 2^p\|f^{(p)}\|_{s-p}.
\end{eqnarray} 
Integration by parts gives 
\begin{eqnarray*} 
\widehat{f}(0)=\sum_{j=0}^{p-1}\frac{(-2\pi)^j}{(j+1)!}f^{(j)}(1)+ \frac{(-1)^p}{2\pi p!}\int_0^{2\pi} t^p f^{(p)}(e^{it}) \, dt.
\end{eqnarray*} 
and 
\begin{eqnarray} \label{norme2}
|\widehat{f}(0)|\leq \sum_{j=0}^{p-1}\frac{(2\pi)^j}{(j+1)!}|f^{(j)}(1)|+ \frac{{(2\pi)^p}}{ (p+1)!}\|f^{(p)}\|_{s-p}.
\end{eqnarray}
Now the lemma follows from (\ref{norme1}) and (\ref{norme2}).
\end{proof}

\begin{lemme} \label{limitpoly1}  Let $s$ be a nonnegative real number and $f$ a function in $A_{s}(\bbt)$ such that $f^{(j)}(1)=0$ for $0\leq j\leq [s]$. Then
\begin{eqnarray*}
\lim_{N \rightarrow +\infty} \big\| f - f_{N,s} \big\|_{s} = 0.
\end{eqnarray*}
\end{lemme}

\begin{proof} We set $p=[s]$. The case $p=0$ is exactly Lemma \ref{limitpoly}. So assume that $p\geq 1$.  Let $f\in A_{s}(\bbt)$ such that $f^{(j)}(1)=0$ for $0\leq j\leq p$. We note that for every $N\geq 1$, $(f_{N,s})^{(p)}=(f^{(p)})_{N,0}$ and $(f_{N,s})^{(j)}(1)=0$ for $0\leq j\leq p$.  By Lemma \ref{norme},
\begin{eqnarray*} 
\| f -f_{N,s}\|_{s} & \leq & \left(2^p+\frac{{(2\pi)^p}}{ (p+1)!}\right) \|f^{(p)}-(f_{N,s})^{(p)}\|_{s-p} \\
& = & \left(2^p+\frac{{(2\pi)^p}}{ (p+1)!}\right) \|f^{(p)}-(f^{(p)})_{N,0}\|_{s-p}. 
\end{eqnarray*}
It follows from Lemma \ref{limitpoly} that $\lim_{N \rightarrow +\infty} \|f^{(p)}-(f^{(p)})_{N,0}\|_{s-p}=0$, which finishes the proof.

\end{proof}



Let $s$ be a nonnegative real number, $E$ be a closed set of $\bbt$ whose boundary ($\mathrm{bdry}(E)$), is finite or countable  and  $f$ be a function  in $A_s (\bbt)$ such that $f=f'=\ldots =f^{(p)}=0$ on $E$, where $p=[s]$. We set $I=\{g\in A_s (\bbt);\ gf\in J_s (E)\};$ $I$ is a closed ideal of $A_s (\bbt)$, $J_s (E)\subset I$ and $h(I)\subset h(J_s (E))=E$. Let $\zeta\in \overset{\circ}{E}$, where $\overset{\circ}{E}$ is the interior of $E$. Since $A_s (\bbt)$ is a regular algebra there exists a function $g\in A_s (\bbt)$ whose support is contained in $\overset{\circ}{E}$ and such that $f(\zeta)=1$. We have $fg=0$ which implies that $g\in I$ and then $\zeta\notin h(I)$. It follows that $h(I)\subset \mathrm{bdry}(E)$.  Assume that $h(I)$ is nonempty and take $\zeta$ an isolated point in $h(I)$. There exists a function $h\in A_s (\bbt)$ such that $h(\zeta)=1$ and $h=0$ on a neighborhood of 
$h(I)\setminus \{\zeta\}$. For $n\geq 1$, we set
 $$
u_{n} = \bigg( \frac{\alpha-\zeta}{\alpha - \zeta \big( 1+\frac{1}{n} \big)} \bigg)^{[s]+1} \qquad (n \geq 1),
$$
 According to  \cite[Proposition 2.4]{Agr1}, we have $\dsp \lim_{n \rightarrow +\infty} \|u_{n}hf - hf\|_{s} = 0$, since $(hf)^{(j)}(\zeta)=0$ for $0\leq j\leq [s]$.
Moreover, for each $n\geq 1$, $u_n\in J_s(\{\zeta\}$ (see \cite[Proposition 6]{Atz} or \cite[Th\'eo\`eme 3.2]{Agr1}). So there exists a function $v_n\in A_s(\bbt)$, vanishing in a neighborhood of $\zeta$ and such that $\|u_n-v_n\|_s\leq 1/n$. Now $v_nh$ vanishes on a neighborhood of $h(I)$ and then $v_nh\in I$. Therefore $v_nhf\in J_s (E)$ and since $\dsp \lim_{n \rightarrow +\infty} \|v_{n}hf - hf\|_{s} = 0$, we have $hf\in J_s(E)$ and then $h\in I$. This contradicts that $h(\zeta)\ne 0$, which implies that $h(I)=\emptyset$ and $I=A_s (\bbt)$. We deduce that $f\in J_s(E)$. Notice that for $s=0$, it is well known that in the Wiener algebra $A(\bbt)$, every  closed subset of $\bbt$ whose boundary is a finite or countable set, satisfies the spectral synthesis (see \cite[Theorem, p. 245]{Katz}).

Suppose now that $E$ is  a finite union of closed arcs (not reduced to a single point). Then $\mathrm{bdry}(E)$ is a finite set  and if $f$ is a function in $A_s(\bbt)$ which vanishes on $E$, then all its derivatives of order less or equal to $[s]$, also vanish on $E$. We deduce that $E$ satisfies spectral synthesis in $A_s (\bbt)$. This allows us to prove the following lemma, which 
was obtained in \cite{Agr3} by intricate computations.

\begin{lemme} \label{synthesecantor}
Let $s$ be a nonnegative real number and $q \geq 3$ an integer. Then $E_{1/q}$ satisfies the spectral synthesis in $A_{s}(\bbt)$. 
\end{lemme}

\begin{proof}  For all $n \geq 1$, we set
$$ 
F_{n} = \bigcup_{(\epsilon_{1}, \ldots, \epsilon_{n}) \in \{0, 1\}^{n} } L_{\epsilon_{1}, \ldots, \epsilon_{n}},
$$
where $L_{\epsilon_{1}, \ldots, \epsilon_{n}}$ is the closed arc $\Big\{ e^{i t} : \, t \in \big[ \dsp 2\pi (q-1) \sum_{k=1}^{n} \epsilon_{k} q^{-k}, 2\pi (q-1) \big(\sum_{k=1}^{n} \epsilon_{k} q^{-k} + q^{-n}\big) \big] \Big\}$. We have $\dsp E_{1/q} = \bigcap_{n \geq 1} F_{n}$ (see  \cite[Chapter I]{KaSa}).

Let $f \in I_{s}(E_{1/q})$, we have to show that $f \in J_{s}(E_{1/q})$. Notice that since $E_{1/q}$ is a perfect set,  $f$ vanishes, as well as all its derivatives of order less or equal to  $[s]$, on $E_{1/q}$. We observe that for each $n\geq 1$ and each $(\epsilon_{1}, \ldots, \epsilon_{n}) \in \{0, 1\}^{n}$, the end points of the arc $L_{\epsilon_{1}, \ldots, \epsilon_{n}}$ are of the form $\frac{2\pi k}{q^n}$ and $\frac{2\pi (k+1)}{q^n}$ for some integer $k\in \{0,1,\ldots, q^n-1\}$, and since $(f_{q^n,\epsilon})^{([s])}=f^{([s])}=0$ at these end points, $f_{q^n,\epsilon}$ vanishes on $L_{\epsilon_{1}, \ldots, \epsilon_{n}}$.
Thus $f_{q^n,\epsilon}$ vanishes on $F_{n}$. Since $F_n$ is a finite union of arcs, it satisfies spectral synthesis in $A_{s}(\bbt)$. Therefore, each function $f_{q^{n}}$ is in $J_{s}(E_{{1}/{q}})$. As $J_{s}(E_{1/q})$ is closed in $A_{s}(\bbt)$, we deduce from Lemma~\ref{limitpoly1} that $f \in J_{s}(E_{1/q})$. 
\end{proof}

Now we are able to prove the  main result of this section. 

\begin{theo} \label{TheorSynthese}
Let $s$ be a nonnegative real number, $q \geq 3$ an integer and $\omega$ a weight such that $$
\dsp \omega(n) = (1+n)^{s}, \ \ n \geq 0,
$$
$$
\dsp \omega(-n) \geq c\ (1+n)^{s} \ \ n \geq 0, \ \ \textrm{for some constant}\ \ \dsp c>0,
$$
and
$$\dsp { \omega(-n)} = O \big( e^{n^{\beta}} \big),\ \ n\to +\infty, 
\ \ \textrm{for some constant}\ \ \dsp \beta <  b \big(1/q\big).$$
Then $E_{{1}/{q}}$ satisfies spectral synthesis in $A_{\omega}(\bbt)$.
\end{theo}

\begin{proof}
We denote by $\pi_{\omega}$ the canonical surjection from $A_{\omega}(\bbt)$ to $A_{\omega}(\bbt) / J_{\omega}(E_{1/q})$, and by $T$ the operator of multiplication by $\pi_{\omega}(\alpha)$ on $A_{\omega}(\bbt) / J_{\omega}(E_{1/q})$. It is easy to see that $T$ satisfies the hypothesis of  Theorem \ref{operateur}. It follows that $\big\| \pi_{\omega}(\alpha)^{-n} \big\| =\big\| T^{-n} \big\| = O \big( n^{s} \big), \, n \rightarrow +\infty$. Thus we can define a continuous morphism $\Theta$ from $ A_{{s}}(\bbt)$ into $A_{\omega}(\bbt) / J_{\omega}(E_{1/q})$ by
$$
\Theta(f) = \sum_{n=-\infty}^{+\infty} \widehat{f}(n) \pi_{\omega}(\alpha)^{n}, \qquad  f \in A_{{s}}(\bbt) .
$$
We note that $\textrm{Ker} \, \Theta$ is a closed ideal of $A_{{s}}(\bbt)$ and that $J_{{\omega}}(E_{1/q})=\textrm{Ker} \, \Theta \cap A_{\omega}(\bbt)$, since for $f\in  A_{\omega}(\bbt)$, $\Theta (f)=\pi_{\omega}(f)$. It follows that $h(\textrm{Ker} \, \Theta) \subset E_{1/q})$ and that $J_{{s}}(E_{1/q})\subset \textrm{Ker} \, \Theta $.
 By Lemma \ref{synthesecantor}, $E_{1/q}$ satisfies spectral synthesis in $A_{{s}}(\bbt)$, that is $I_{{s}}(E_{\frac{1}{q}}) = J_{{s}}(E_{1/q})$. Therefore $I_{{\omega}}(E_{1/q})=I_{{s}}(E_{1/q}) \cap A_{\omega}(\bbt)=J_{{s}}(E_{1/q}) \cap A_{\omega}(\bbt)\subset \textrm{ker} \, \Theta \cap A_{\omega}(\bbt)=J_{\omega}(E_{1/q})$, which completes the proof.
\end{proof}

\end{document}